\documentclass[12pt]{amsart}

\usepackage{amsmath, amsfonts, amssymb, epsfig, amscd, graphicx, latexsym, latexsym, amsthm, etex, synttree, tikz}
\usepackage[latin5]{inputenc}
\usepackage{young}

\newcommand{\db}{\mbox{\rm {{\tiny d-braid}}}}

\newcommand{\beeq}{\begin{equation}}
\newcommand{\eneq}{\end{equation}}
\newcommand{\bece}{\begin{center}}
\newcommand{\ence}{\end{center}}

\newcommand{\mb}{\mathbf}

\newtheorem{thm}{Theorem}[section]
\newtheorem{pro}[thm]{Proposition}

\newtheorem{lem}[thm]{Lemma}

\theoremstyle{definition}
\newtheorem{ex}[thm]{Example}
\newtheorem{defn}[thm]{Definition}
\newenvironment{rem}[1][Remark.]{\begin{trivlist}
\item[\hskip \labelsep {\bfseries #1}]}{\end{trivlist}}

\newcommand{\cellsize}{10}
\newlength{\cellsz} 
\newsavebox{\cell}
\sbox{\cell}{\begin{picture}(\cellsize,\cellsize)
\put(0,0){\line(1,0){\cellsize}}
\put(0,0){\line(0,1){\cellsize}}
\put(\cellsize,0){\line(0,1){\cellsize}}
\put(0,\cellsize){\line(1,0){\cellsize}}
\end{picture}}
\newcommand\cellify[1]{\def\thearg{#1}\def\nothing{}%
\ifx\thearg\nothing
\vrule width0pt height\cellsz depth0pt\else
\hbox to 0pt{\usebox{\cell} \hss}\fi%
\vbox to \cellsz{
\vss
\hbox to \cellsz{\hss$#1$\hss}
\vss}}
\newcommand\tableau[1]{\vtop{\let\\\cr
\baselineskip -16000pt \lineskiplimit 16000pt \lineskip 0pt
\ialign{&\cellify{##}\cr#1\crcr}}}

\begin{document}

\title{Sorting and generating reduced words}

\author{Olcay Coşkun} \author{Müge Taşkın}\thanks{Both of the authors are supported by Boğaziçi - Bap - 6029.}
\address{Boğaziçi Üniversitesi, Matematik Bölümü, Bebek, İstanbul, Turkey}
\email{olcay.coskun@boun.edu.tr, muge.taskin@boun.edu.tr}
\date{10, 2012}
\begin{abstract}
We introduce a partial order on the set of all reduced words of a
given permutation $\omega$, called \emph{directed-braid poset} of $\omega$. This
poset enables us to produce two algorithms: One is a sorting
algorithm applied on any reduced word of $\omega$ and aims to obtained
the natural word (lexicographically largest reduced word); the other
one is a generation algorithm applied on the natural word and aims
to obtained the set of all reduced words of $\omega$.
\end{abstract}
\maketitle

\section{Introduction}

The symmetric group  $S_n$  on $[n]= \{1,2,\ldots, n\}$ is generated
by adjacent transpositions $\{ s_i: i= 1,2,\ldots n-1\}$, where
$s_i$ stands for the transposition $(i,i+1)$. Given any permutation
$\omega\in S_n$, and any expression $s_{i_1}s_{i_2}\ldots s_{i_l}$
representing $\omega$, we call the sequence  $i_1i_2\ldots i_l$ a
\textit{word} for $\omega$. Such an expression with minimal $l$ is
called a \textit{reduced word}  for $\omega$ and  $l$ is called as
the length of $\omega$, denoted by $\mathrm
{l}(*omega)$. The index $l$ is
determined by $\omega$ and, by a result of Tits, any two
reduced words for $\omega$ are related by the braid relations given
as follows.
\begin{enumerate}
\item (Short braid relation) $s_is_j = s_js_i$ for any $i,j$ with $\vert i-j\vert \ge 2$.
\item (Long braid relation) $s_is_{i+1}s_{i} = s_{i+1}s_is_{i+1}$ for any $i$.
\end{enumerate}

For any permutation $\omega$, we denote by $\mathcal{R}(\omega)$ the set of
all reduced words of $\omega$. By the above result, any reduced word can be
obtained from any other by applying a series of braid relations. However this
procedure of applying a sequence of braid relations is non-trivial
since one needs to go back and forth using the braid relations. For that reason, we introduce the
\textit{directed-braid poset} $(\mathcal{R}(\omega),<_{\mathrm{d-braid}})
$ where the underlying partial order is obtained by putting a
direction on the applications of the short and long braid relations.

The partial order $<_{\mathrm{d-braid}}$ is weaker than the lexicographic order on $\mathcal R(\omega)$.
Moreover the unique maximal word in $(\mathcal R(\omega), \le_{\mathrm lex})$ remains as the unique
maximal element in the directed-braid poset. This word is the one that we obtain when we apply the selection sort 
algorithm to the one line notation of the permutation. The algorithm aims to convert the permutation to the identity,
by moving, at each step, the smallest number in $\omega$ which is not in its identity position.

The maximal element can also be described by tower diagrams introduced in \cite{CT}, where it is called the 
\emph{natural word} of the permutation. We will use the notation of \cite{CT}. By \cite[Proposition 5.1]{CT}, the
natural word can be characterized as a list of increasing sequences of consecutive integers with decreasing first
 terms.  
 
With the direction provided by directed-braid poset, we then
introduce a sorting algorithm on any reduced word of $\omega$, which at
the end obtains the natural word of $\omega$ as chain in directed-braid
poset. Thus we obtain a canonical route between any two reduced
words of $\omega$ which goes through the natural word. This algorithm can be seen as the combination of
selection sort and insertion sort algorithms.

Having the canonical
route provided by the sorting algorithm, we then introduce  a
generation algorithm for all reduced words of given permutation,
starting from its natural word.  At the first step, we produce
\textit{basic words} of the given permutation. This step is
basically the application of the long braid relation to the natural
word. After determining the set of all basic words, it only remains
to apply the short braid relation, which is done by the
\emph{restricted shuffle} on basic words. This step is a variation
of the well-known shuffle operation on words.

There are combinatorial objects  to determine the set of all reduced words of a given permutation, such as
balanced labeling of the Rothe diagram \cite{FGRS} and RC graphs \cite{BB} of permutations, the
plactification map \cite{RS} and the tower tableaux \cite{CT}. Generating reduced words from these objects is not efficient. Two examples of
 more efficient algorithms are the one that counts saturated chains in the weak order on permutations and the one
 which uses heaps of reduced words \cite{JS}. See \cite{G} for more information on these algorithms.
We leave the comparison of our algorithm, in terms of computational
complexity, with these algorithms as a problem to the interested
reader.

The paper is organized as follows. In Section \ref{Section:poset}, we
introduce the directed-braid poset together with its basic
properties. The sorting  and the generation algorithms are
introduced in \ref{Section:sorting} and \ref{Section:generation}
respectively. To illustrate our algorithm, we produce the reduced words of the
longest permutation in $S_4$.

The following notations will be used throughout the paper.
Let $\mathfrak{b} = \beta_1\beta_2\ldots \beta_l$ be a word with
$l\ge 1$. We call $\mathfrak{b}$ a \textit{tower word} if for any
index $1<i\le l$, we have $\beta_i = \beta_{i-1}+1$; that is, if
$\mathfrak{b}$ is an increasing sequence of consecutive integers.
 It is clear, from the construction in the above discussion, that if $\mathfrak{b}$ is a tower word, then it is the natural word
 of a tower diagram with a unique tower of positive length. For example, $4567$ is a tower word.

It is also clear that any word $\beta =\beta_1\beta_2\ldots \beta_n$
can be written uniquely as the concatenation of tower words, say as
$\beta= \mathfrak{b}_1\mathfrak{b}_2\ldots \mathfrak{b}_s$ for some
$s$, where $\mathfrak{b}_i$ is called $i$-th tower word in $\beta$.
We call $\mathfrak{b}_1\mathfrak{b}_2\ldots \mathfrak{b}_s$ as the
\textit{tower decomposition} of $\beta$. For example, if $\beta =
789 5 453456 2$, then the tower decomposition is given by
$$\beta = \underset{\mathfrak{b}_1}{\underbrace{789}}\underset{\mathfrak{b}_2}{\underbrace{5}} \underset{\mathfrak{b}_3}
{\underbrace{45}}
\underset{\mathfrak{b}_4}{\underbrace{3456}}\underset{\mathfrak{b}_5}{\underbrace{2}}.$$

\section{The directed-braid poset of a permutation}\label{Section:poset}
In this section, we introduce the directed-braid poset  of a
permutation $\omega$ as a main tool for constructing a sorting
algorithm on the set $\mathcal{R}(\omega)$ of all reduced words of
$\omega$.  This poset will be the combination of two relations
defined as follows.
\begin{defn}
Let $\alpha,\beta\in \mathcal{R}(\omega)$ be two reduced words and
$l=\mathrm{l}(\omega)$. We write $\alpha <_1 \beta$ if there exist
$1\leq i < l$ such that
$$\begin{aligned}
&\alpha=\alpha_1\ldots \alpha_{i} \alpha_{i+1}\ldots \alpha_{l}, \\
&\beta=\alpha_1\ldots \alpha_{i+1} \alpha_{i}\ldots \alpha_{l} \text{ and}\\
&\alpha_{i+1}\geq \alpha_i+2. \end{aligned}$$
\end{defn}
Clearly this covering relation refers to the short braid relation and also puts a restriction on the direction that we
can apply it. A key property of this relation is the following lemma.
\begin{lem}\label{lem:1vslex}
The relation $<_1$ implies the lexicographic order, that is, if
$\alpha<_1 \beta$, then we also have $\alpha\le_{lex}\beta$.
\end{lem}
The proof follows easily from the definition. It is trivial that the
converse is not true. For example $121\le_{lex} 212$ but $121\not<_1
212$. As a corollary to this lemma, we get that the reflexive   and
transitive closure of the relation generated by $<_1$ is a partial
order. Indeed, the only non-trivial property is that of
anti-symmetry which is guaranteed by the above lemma.

In order to define the second relation, we will first introduce
several notations: Let $\mathfrak{a}$ be a tower word. Then we
denote by
\begin{enumerate}
\item[i.] $\widetilde{\mathfrak{a}}$ the tower word obtained by
increasing the numbers in the tower word $\mathfrak{a}$ by $1$
\item[ii.] $\mathrm{in}(\mathfrak{a})$  the initial letter of
$\mathfrak{a}$.
\item[iii.] $\mathrm{fin}(\mathfrak{a})$  the final letter of
$\mathfrak{a}$.
\end{enumerate}

The definition of the second relation which  depends the tower
decomposition of reduced words is as follows:
\begin{defn} Let $\alpha, \beta \in \mathcal{R}(\omega)$ and let
$\alpha=\mathfrak{a}_1\ldots \mathfrak{a}_i\mathfrak{a}_{i+1}\ldots
\mathfrak{a}_s$ be the tower decomposition of $\alpha$. We say
$\alpha<_2 \beta$ if there exist $1\leq i < s$ such that
$$\begin{aligned}
&\mathrm{in}(\mathfrak{a}_i)\leq
\mathrm{in}(\mathfrak{a}_{i+1})<\mathrm{fin}(\mathfrak{a}_{i+1})<\mathrm{fin}(\mathfrak{a}_i)
\\
\text{ and }& \beta=\mathfrak{a}_1\ldots
\mathfrak{a}_{i-1}\widetilde{\mathfrak{b}_1}\mathfrak{a}_i\mathfrak{b}_2\mathfrak{a}_{i+2}\ldots
\mathfrak{a}_s
 \end{aligned}
$$
where the tower words $\mathfrak{b}_1$ and (possibly empty)
$\mathfrak{b}_2$,  satisfy that
$$\mathfrak{b}_1\mathfrak{b}_2=\mathfrak{a}_{i+1}.$$
\end{defn}

\begin{rem} First observe that
the representation of $\beta$ in the above definition is not
necessarily the tower decomposition of $\beta$, since
$\mathfrak{a}_{i-1}\widetilde{\mathfrak{b}_1}$ might already be a
tower word.

Secondly  the condition that $\mathrm{in}(\mathfrak{a}_i)\leq
\mathrm{in}(\mathfrak{a}_{i+1})<\mathrm{fin}(\mathfrak{a}_i)$
necessarily implies
$\mathrm{fin}(\mathfrak{a}_{i+1})<\mathrm{fin}(\mathfrak{a}_i)$,
since the other case yields a contradiction to the fact that
$\alpha$ is a reduced word.
\end{rem}

As an example, consider the following three reduced words
$$\alpha = 23\,\,
5678\,\, 67,~~ \beta = 23\,\, 78\,\, 5678 \text{ and } \gamma =
23\,\, 7\,\, 5678\,\,  7$$ of the  permutation $w= 134268975$ given
by their tower decompositions. Then we have
$\alpha <_2
\beta \text{ and } \alpha<_2\gamma \text{ and also } \gamma<_2 \beta.
$

It is clear that via $<_2$, we are moving tower words from right to
left and letter by letter. During these moves, the condition is
basically given by  a series of braid relations which always include
a long one.  As in the previous case, we have the following lemma.

\begin{lem}\label{lem:2vslex}
The relation $<_2$ implies the lexicographic order, that is, if
$\alpha<_2 \beta$, then we also have $\alpha\le_{lex}\beta$.
\end{lem}
Again, the proof follows from the definition and clearly the
converse is not true. For example $12 4\le_{lex} 142$ but
$124\not<_2 142$. Moreover, the reflexive and the transitive closure
of the relation generated by $<_2$ is a partial order.

Now we define a partial order on the set of all reduced words.
\begin{defn} For  $\alpha, \beta \in \mathcal{R}(\omega)$, we write
$$\alpha  \leq_{\db} \beta$$ if
either $\alpha = \beta$ or there is a sequence
$\gamma^0,\gamma^1,\ldots,\gamma^m$ of reduced words in
$\mathcal{R}(\omega)$ such that $\gamma^0 = \alpha, \gamma^m = \beta$ and
for any $i, 0\le i \le m-1$, we have either $\gamma^i<_1
\gamma^{i+1}$ or $\gamma^i<_2 \gamma^{i+1}$.
\end{defn}

 We have the following result.

\begin{pro}
The  relation $\le_{\db}$ on $\mathcal R(\omega)$ is a partial
order.
\end{pro}
\begin{proof}
Reflexivity and transitivity of the relation follows directly from the definition. We only prove that the relation
is anti-symmetric. By Lemma \ref{lem:1vslex} and Lemma \ref{lem:2vslex}, if $\alpha \lneq_i\beta$,
for $i=1,2$ then $\alpha\lneq_{lex} \beta$. Thus $\le_{\db}$ is anti-symmetric as $\le_{lex}$ is.
\end{proof}

The main result of this section is that the poset  $(\mathcal
R(\omega),\le_{\db})$ has a unique maximal element. As we have
described above, this result is the starting point of an algorithm
to generate reduced words from the natural one.
\begin{pro} \label{maxword}
The natural word $\eta_\omega$ of $\omega$ is the unique maximal
element in $(\mathcal R(\omega),\le_{\db})$.
\end{pro}

\begin{proof}
Let $\alpha$ be a maximal element in $(\mathcal
R(\omega),\le_{\db})$ and let $$\alpha=\mathfrak{a}_1\ldots
\mathfrak{a}_i\mathfrak{a}_{i+1}\ldots \mathfrak{a}_k$$ be the tower
decomposition of $\alpha$. Now since $\alpha$ is maximal, there is
no word $\beta$ such that $\alpha <_1 \beta$. But this is only
possible if for any $i,1\le i< k$, we have
$$
\mathrm{fin}(\mathfrak{a}_i) > \mathrm{in}(\mathfrak{a}_{i+1}).
$$
 Indeed, otherwise, $
\mathrm{fin}(\mathfrak{a}_i) < \mathrm{in}(\mathfrak{a}_{i+1}) $
yields  $
 \mathrm{in}(\mathfrak{a}_{i+1})-\mathrm{fin}(\mathfrak{a}_i) >2$
 since $\mathfrak{a}_{i+1}$ and $\mathfrak{a}_{i}$ are different
tower words. Hence one can  interchange
$\mathrm{fin}(\mathfrak{a}_i)$ and $\mathrm{in}(\mathfrak{a}_{i+1})$
to obtain a greater word in $(\mathcal R(\omega),\le_{\db})$.

Now suppose that for some $1\le i< k$ we have
$\mathrm{in}(\mathfrak{a}_{i}) \leq
\mathrm{in}(\mathfrak{a}_{i+1})$. Then by the above discussion we
have
$$\mathrm{in}(\mathfrak{a}_{i}) \leq
\mathrm{in}(\mathfrak{a}_{i+1})< \mathrm{fin}(\mathfrak{a}_i)$$ but
this forces that $\mathrm{fin}(\mathfrak{a}_{i+1}) \leq
\mathrm{fin}(\mathfrak{a}_{i})$, since otherwise $\alpha$ is not a
reduced word. Hence by definition of $<_2$, one can  obtained a
reduced word $\beta$ such that $\alpha <_2 \beta$, but this
contradicts to the maximality of $\alpha$. Therefore for each $1\le
i< k$, we  have
$$\mathrm{in}(\mathfrak{a}_{i}) >
\mathrm{in}(\mathfrak{a}_{i+1}).$$

Thus we have proved that any maximal element $\alpha$ in $(\mathcal
R(\omega),\le_{\db})$ has the property that in its  tower
decomposition,  the sequence of initial letters of tower words is
decreasing. But by
 \cite[Proposition 5.1]{CT}, there is a unique
word with this property, namely the natural word.
\end{proof}

Although there is a unique maximal in the braid poset, there might be many minimal elements. An example of a braid
poset with two minimal elements is the poset of the word $432123$, where the minimals are the words $124321$ and
$143213$. The full poset in this case is given as follows.
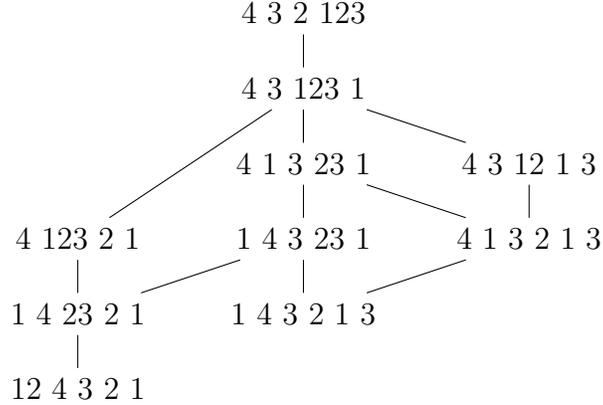
\begin{center}
\begin{tikzpicture}
  \node (max) at (0,5) {$4~3~2~123$};
  \node (a) at (0,4) {$4~3~123~1$};
  \node (b1) at (0, 3){$4~1~3~23~1$};
  \node (c2) at (0,2){$1~4~3~23~1$};
  \node (d2) at (0,1){$1~4~3~2~1~3$};
 \node (c3) at (3, 2){$4~1~3~2~1~3$};
 \node (b2) at (3, 3){$4~3~12~1~3$};
\node (d1) at (-3,1){$1~4~23~2~1$};
\node (c1) at (-3,2){$4~123~2~1$};
  \node (e) at (-3,0){$12~4~3~2~1$};
\draw  (max) -- (a) -- (b1)
--(c2)--(d2)--(c3)--(b2)--(a)--(b1)--(c3)(a)--(c1)--(d1)--(e)(c2)--(d1);
\end{tikzpicture}
\end{center}

\begin{rem}[Remark 1.]
The above result tells us that it is possible to generate all
reduced words starting from the natural word. The problem is to
determine a canonical path from the maximal element to the chosen
one. The advantage we have here is that via the directed-braid
poset, we insist a direction on the braid relations: The short braid
relation is the equality $s_is_j = s_js_i$ if $|j-i| \ge 2$, but in
the braid poset, to go from the natural word to an arbitrary word,
we can only interchange $j$ and $i$ if $j$ is smaller than $i$.
Similar comment is true for the long braid relation. Therefore, a
maximal element in this poset is a word on which the directed-braid
relations cannot be applied. In this sense, the natural word is the
only braid-free word. Hence one would expect to have a canonical
route from the natural word to any other reduced word and vice
versa. For the rest of the paper, we explain such canonical routes.
\end{rem}

\begin{rem}[Remark 2.] In \cite{KLR}, the poset of commutation classes is introduced. Given a permutation $\pi$,
two reduced words $\alpha$ and $\beta$ are in the same commutation
class if they differ from each other by the short braid relation,
hence are comparable with the partial order generated by $<_1$. This
poset is used in \cite{KLR} in relation with factorization of
Schubert cells.
\end{rem}
\section{Sorting algorithm: From a reduced word to the natural
word of a permutation}\label{Section:sorting}

By the Proposition~\ref{maxword},   the natural word of a
permutation $\omega$ is the unique maximum among all reduced words of $\omega$
in the directed-braid poset.  In this section, we introduce two
algorithms, by which one obtains the natural word $\eta_\omega$  from any
reduced word $\alpha$ of $\omega$ as a chain in this poset.

\subsection{A selection sort algorithm on reduced words}

Let $\alpha$ be a word reduced word for a permutation $\omega$, given
with its tower decomposition $\alpha=\mathfrak{a}_1\ldots
\mathfrak{a}_i\mathfrak{a}_{i+1}\ldots \mathfrak{a}_k$. We call
$\alpha$  a \textit{natural basic word} for $\omega$ if for each $1\leq i<k$ we have
$$\mathrm{fin}(\mathfrak{a}_i)>\mathrm{in}(\mathfrak{a}_{i+1}).
$$

It is easy to see the natural word of  $\omega$ is the unique natural basic word which  satisfies that  the sequence of
initial letters of its  tower words is strictly decreasing.

The selection sort algorithm on any reduced word $\alpha$ of $\omega$ aims to obtain a natural basic word, say 
$\beta$ of $\omega$,  as a chain
$$\alpha=\beta^0\leq_1 \beta^1\leq_1 \ldots \beta^k=\beta$$
of  reduced words in $(\mathcal R(\omega),\le_{\db})$.

We describe the algorithm inductively. Suppose that $\beta^j$ is
constructed.  Then we obtain $\beta^{j+1}$ as follows: Write
$$\beta^j=\mathfrak{b}_1\ldots \mathfrak{b}_r$$ in its  the tower
decomposition  and $a$ let be its smallest letter. Then $a$ is
necessarily an initial letter of some tower words in $\beta^j$ and
let $\mathfrak{b}_s$ be  the right most tower word starting with
$a$, for some $1\leq s \leq r$.

\begin{enumerate}
\item[i.]  If $s<r$ and if  $\mathrm{fin}(\mathfrak{b}_s)<
\mathrm{in}(\mathfrak{b}_{s+1})$ (necessarily
$\mathrm{in}(\mathfrak{b}_{s+1})- \mathrm{fin}(\mathfrak{b}_s)\geq
2$) then we let
$$\beta^{j+1}= \mathfrak{b}_1\ldots\mathfrak{b}_{s+1}\mathfrak{b}_{s}\ldots
\mathfrak{b}_r.
$$
\item[ii.] If $s=r$ or if $\mathrm{fin}(\mathfrak{b}_s)>
\mathrm{in}(\mathfrak{b}_{s+1})$ then  we continue by applying the
same algorithm on the right most tower word   starting with  $a$,
which is on the  left of $\mathfrak{b}_{s}$.
\item[iii.] Finally, if none of the tower words starting with $a$ in $\beta^{j}$ can move to the right  subject to the above
rule,  then we continue to apply the same algorithm with  the
smallest initial letter bigger than $a$ in  $\beta^j$.
\end{enumerate}

Observe that the above algorithm may not produce a new word,
$\beta^{j+1}$, and this happens if and only if  $\beta^j$  is a
natural basic word i.e., for any $1\leq s <r$
$$\mathrm{fin}(\mathfrak{b}_s)> \mathrm{in}(\mathfrak{b}_{s+1})$$
In this case we let $\beta=\beta^j$. On the other hand if
$\beta^{j+1}$ is produced as a result then we have
$\beta^j<_1\beta^{j+1}.$

\begin{ex} We consider the reduced word  $\alpha= 134521321654321$ to produce  a natural basic word.   We use brackets to indicate
 the tower words which are to be moved  according to the
above algorithm.
$$\begin{aligned}
\alpha=\alpha^0=& 1~~34567~~2~~1~~3~~2~~\textbf{[1]}~~6~~45~~4~~3~~2~~1\\
<_1\beta^1=& 1~~34567~~2~~1~~3~~2~~6~~\textbf{[1]}~~45~~4~~3~~2~~1\\
<_1\beta^2=& 1~~34567~~2~~1~~3~~2~~6~~45~~\textbf{[1]}~~4~~3~~2~~1\\
<_1\beta^3=& 1~~34567~~2~~1~~3~~2~~6~~45~~4~~\textbf{[1]}~~3~~2~~1\\
<_1\beta^4=& 1~~34567~~2~~\textbf{[1]}~~3~~2~~6~~45~~4~~3~~12~~1\\
<_1\beta^5=& 1~~34567~~23~~\textbf{[12]}~~6~~45~~4~~3~~12~~1\\
<_1\beta^6=& 1~~34567~~23~~6~~\textbf{[12]}~~45~~4~~3~~12~~1\\
<_1\beta^7=& 1~~34567~~23~~6~~45~~\textbf{[12]}~~4~~3~~12~~1\\
<_1\beta^8=& \textbf{[1]}~~34567~~23~~6~~45~~4~~123~~12~~1\\
<_1\beta^9=& 34567~~\textbf{[123]}~~6~~45~~4~~123~~12~~1\\
<_1\beta^{10}=& 34567~~6~~12345~~4~~123~~12~~1=\beta.
\end{aligned}
$$

\end{ex}

\subsection{An insertion sort algorithm on natural basic words}\label{Section:AlgorithmOn2}
Our next aim is to construct an algorithm which transforms a natural
basic word of a permutation to its  natural word. The algorithm is
based on the relation $<_2$.

Let $\beta=\mathfrak{b}_1\ldots \mathfrak{b}_k$ be a natural basic
word for a permutation, given with its tower decomposition.
Therefore
 $$ \mathrm{fin}(\mathfrak{b}_{i})>
\mathrm{in}(\mathfrak{b}_{i+1}) ~~\mathrm{ for ~all }~~ 1\leq i<k.$$
Note that, in this case we have
$$\begin{aligned}
\mathrm{ either }&~~ \mathrm{in}(\mathfrak{b}_{i})>
\mathrm{in}(\mathfrak{b}_{i+1})\\
\mathrm{ or }&~~ \mathrm{in}(\mathfrak{b}_{i})\leq
\mathrm{in}(\mathfrak{b}_{i+1})<\mathrm{fin}(\mathfrak{b}_{i}).
\end{aligned}$$

Observe that the second case also forces that
$$\mathrm{in}(\mathfrak{b}_{i})\leq
\mathrm{in}(\mathfrak{b}_{i+1})\leq
\mathrm{fin}(\mathfrak{b}_{i+1})<\mathrm{fin}(\mathfrak{b}_{i})
$$
since otherwise  $\beta$ can not be a reduced word. Moreover in this
case the two words
$$
\beta=\mathfrak{b}_1\ldots\mathfrak{b}_{i}\mathfrak{b}_{i+1}\ldots
\mathfrak{b}_k ~~\mathrm{and}~~
\beta'=\mathfrak{b}_1\ldots\widetilde{\mathfrak{b}_{i+1}}\mathfrak{b}_{i}\ldots
\mathfrak{b}_k
$$
 can be
obtained from one another through  a sequence of short and long
braid relations and $\beta<_2 \beta'$.

Now we are  ready to explain the algorithm which converts any
natural basic word  to the unique the natural word of the
corresponding permutation. Let $\beta$ be a natural basic word whose
tower word decomposition is of the following form.
$$\beta=\beta^0=\mathfrak{b}_1\ldots \mathfrak{b}_k.
$$
If $\mathrm{in}(\mathfrak{b}_i) > \mathrm{in}(\mathfrak{b}_{i+1})$
for all $1\leq i \leq k$ then $\beta$ is the natural word and we do
not proceed. Otherwise let $j$ be the largest index such that
$\mathrm{in}(\mathfrak{b}_i) \leq \mathrm{in}(\mathfrak{b}_{i+1})$.
As it is discussed above, this forces  that
$$\mathrm{in}(\mathfrak{b}_{i})\leq
\mathrm{in}(\mathfrak{b}_{i+1})\leq
\mathrm{fin}(\mathfrak{b}_{i+1})<\mathrm{fin}(\mathfrak{b}_{i})
$$
and we consider  the following word
\begin{equation}\label{naturalbasicword}
\mathfrak{b}_1\ldots
\mathfrak{b}_{i-1}\widetilde{\mathfrak{b}_{i+1}}\mathfrak{b}_i\ldots
\mathfrak{b}_k \end{equation} which is clearly  is braid equivalent
to $\beta^0$. We have three cases to consider.

\begin{itemize}
\item[i.] If
$\mathrm{fin}(\mathfrak{b}_{i-1})>\mathrm{in}(\widetilde{\mathfrak{b}_{i+1}})$
then $\mathfrak{b}_1\ldots
\mathfrak{b}_{i-1}\widetilde{\mathfrak{b}_{i+1}}\mathfrak{b}_i\ldots
\mathfrak{b}_k$ is the tower decomposition of a natural basic word
and we let  $$\beta^1=\mathfrak{b}_1\ldots
\mathfrak{b}_{i-1}\widetilde{\mathfrak{b}_{i+1}}\mathfrak{b}_i\ldots
\mathfrak{b}_k.$$
\item[ii.]
If
$\mathrm{fin}(\mathfrak{b}_{i-1})+1=\mathrm{in}(\widetilde{\mathfrak{b}_{i+1}})$
 then we naturally concatenate
$\mathfrak{b}_{i-1}$ and
$\widetilde{\mathfrak{b}_{i+1}}$ to get the tower
decomposition of this word. Since
$$\mathrm{fin}(\mathfrak{b}_{i-1}\widetilde{\mathfrak{b}_{i+1}})>\mathrm{in}(\mathfrak{b}_{i}),$$
 we let $\beta^1=\mathfrak{b}_1\ldots
\mathfrak{b}_{i-1}\widetilde{\mathfrak{b}_{i+1}}\mathfrak{b}_i\ldots
\mathfrak{b}_k$ be the natural basic word in
\eqref{naturalbasicword}.
\item[iii.]
If
$\mathrm{fin}(\mathfrak{b}_{i-1})+1<\mathrm{in}(\widetilde{\mathfrak{b}_{i+1}})$
then the word in \eqref{naturalbasicword} is not a natural basic word. But
then one can move $\widetilde{\mathfrak{b}_{i+1}}$ to the left of
$\mathfrak{b}_{i-1})$ by using short braid relations and  continue
in the similar manner, if necessary, until the resulting word is a
natural basic word. In this case we let $\beta^1$ be this natural
basic word.
\end{itemize}

The above algorithm yields that $\beta^0 <_{d-braid} \beta^1$.  Now
applying the above algorithm repeatedly we obtain a sequence of
natural basic words
$$\beta=\beta^0 < \beta^1 \ldots .
$$
This sequence  terminates after finitely many steps at some word,
say $\beta^r$, in which the sequence of  initial letters  of each
tower words is strictly decreasing, i.e $\beta^r$ is the natural
word of the corresponding permutation.

\begin{ex} Observe that $\beta= 2345678~~234~~1234567~~56$ is a
natural basic word given by its tower decomposition. İn the
following we use brackets to indicate the tower words subject to
move according to above algorithm.

$$\begin{aligned}
\beta=\beta^0=& ~ 2345678~~234~~1234567~~[\textbf{56}]~~ \\
<_{d-braid}\beta^1=& ~2345678~~[\textbf{67}]~~234~~1234567~~\\
<_{d-braid} \beta^2=& ~ 78~~2345678~~[\textbf{234}]~~1234567~~\\
<_{d-braid} \beta^3=&  ~78~~345~~2345678~~1234567  \\
\end{aligned}$$

For another example we consider $\gamma=3456~~5~~1234~~123~~12~~1$
which is also a   natural basic word given with its tower
decomposition. Then
$$\begin{aligned}
\gamma=\gamma^0=&3456~~5~~1234~~123~~12~~[\textbf{1}] \\
<_{d-braid} \gamma^1=&~3456~~5~~1234~~123~~[\textbf{2}]~~12 \\
<_{d-braid} \gamma^2=&~3456~~5~~1234~~[\textbf{3}]~~123~~12 \\
<_{d-braid} \gamma^3=&~3456~~5~~4~~1234~~123~~[\textbf{12}] \\
<_{d-braid} \gamma^4=&~3456~~5~~4~~1234~~[\textbf{23}]~~123 \\
<_{d-braid} \gamma^5=&~3456~~5~~4~~34~~1234~~[\textbf{123}] \\
<_{d-braid}\gamma^6=&~3456~~[\textbf{5}]~~4~~34~~234~~1234 \\
<_{d-braid}\gamma^7=&~6~~3456~~[\textbf{4}]~~34~~234~~1234 \\
<_{d-braid}\gamma^8=&~6~~5~~3456~~[\textbf{34}]~~234~~1234 \\
<_{d-braid}\gamma^{9}=&~6~~5~~45~~3456~~234~~1234=\eta \\
\end{aligned}$$
\end{ex}

\section{Generation algorithm: from the natural word to a reduced
word} \label{Section:generation}

The sorting  algorithm introduced in the previous section shows that
there is a canonical route from an arbitrary reduced word to the
natural word. In this section, we will try to reverse this algorithm
to get a generation theorem.

The generation algorithm of the reduced words of any permutation
$\omega$, as the sorting algorithm suggests, consists of two parts.
In the first part, we only allow the tower words of the natural word
of $\omega$ to pass  each other by the \emph{passage operation},
whose definition arises from the insertion sort algorithm.  We call
each word obtained in this way a basic word, in fact some of the
words are natural basic words. In the next step, by taking the
selection sort algorithm  into account, we apply \emph{restricted
shuffle operation} on each basic words to obtain all reduced words
of $\omega$.

\subsection{Basic words of a tower diagram}\label{Subsection:basicwords}
We begin with preliminary definitions.

 Let $b\in \mathbb{Z}^+$ and let
$\alpha=\mathfrak{a}_1,\ldots,\mathfrak{a}_r$ be a reduced word of a
permutation given by its tower decomposition. If
$b_0\mathfrak{a}_1\mathfrak{a}_2\ldots\mathfrak{a}_r$ is a reduced
word then the \textit{track sequence of $b$ through $\alpha$} is
 the largest finite sequence of terms
$$ b_0,b_1\ldots, $$
such that $b_0=b$ and for $i\geq 1$
$$b_{i}:=
 \begin{cases} (b_{i-1})-1 ~~& \text{if}~~ \mathrm{in}(\mathfrak{a}_i)<b_{i-1}\leq \mathrm{fin}(\mathfrak{a}_i) \\
                 b_{i-1} ~~& \text{if either}~~ b_{i-1}\leq \mathrm{in}(\mathfrak{a}_i)-2 ~~ \text{or}~~
                   b_{i-1}\geq \mathrm{fin}(\mathfrak{a}_i)+2, \\
                  \text{undefined}~~ & \text{otherwise.}
\end{cases}
$$
It is clear that if $b_i$ is not defined, for some $1\leq i\leq r$,
then $b_{i+1}$ is not defined. Hence the maximum possible number of
elements in this sequence $r+1$.

\begin{defn} Let $b\in \mathbb{Z}^+$ and let
$\alpha=\mathfrak{a}_1,\ldots,\mathfrak{a}_r$ be a reduced word of a
permutation given by its tower decomposition. If
$b_0\mathfrak{a}_1\mathfrak{a}_2\ldots\mathfrak{a}_r$ is not a
reduced word then we set $\mathrm{passwords}(b,\alpha)=\varnothing$.
Otherwise we set
$$\begin{aligned}
\mathrm{passwords}(b,\alpha)&:=\{\alpha^0=b_0\mathfrak{a}_1\mathfrak{a}_2\ldots\mathfrak{a}_r\} \\
&\cup\{\alpha^i=\mathfrak{a}_1\ldots\mathfrak{a}_{i}b_i\mathfrak{a}_{i+1}\ldots\mathfrak{a}_r
\mid 1\leq i\leq s\}
\end{aligned}$$
where $ b_0,b_1\ldots,b_s $ is the track sequence of $b$ through
$\alpha$.
\end{defn}

\begin{ex} Let $\alpha = 3456~ 789~ 12345
$, and $b=b_0= 6$.  Then
$\mathrm{passwords}(b,\alpha)$ consist of the following reduced
words
\begin{eqnarray*}
\alpha^0 &=& [\textbf{6}]~ 3456~ 789~ 12345\\
\alpha^1 &=& 3456~ [\textbf{5}]~ 789~ 12345\\
\alpha^2 &=& 3456~ 789~ [\textbf{5}]\, 12345\\
\alpha^3 &=& 3456~ 789~ 12345 ~[\textbf{4}]
\end{eqnarray*}
where the numbers in the brackets,  gives the track sequence of
$b=6$ through $\alpha$. Namely, $b_0=6, b_1=5, b_2=5,b_3=4.$ One can easily check the following examples.
$$ \begin{aligned}  \mathrm{passwords}(10,\alpha)=& \{10~ 3456~ 789~ 12345,~ 3456~10~ 789~ 12345 \} \\
\mathrm{passwords}(2,\alpha)=& \{23456~ 789~ 12345\}\\
\mathrm{passwords}(1,\alpha)=&\varnothing
\end{aligned}
$$

\end{ex}

\begin{defn} Let $\alpha$ and  $\beta$  be two reduced words such
that the concatenation $\beta\alpha$ is  also reduced. We  define
the set of all \textit{passage words of $\beta$ through $\alpha$} as
follows: Write $\beta=\beta_1\beta_2\ldots \beta_n$, then
$$\begin{aligned}
\mathrm{passwords}(\beta,\alpha)=&\mathrm{passwords}(\beta_1\beta_2\ldots
\beta_n,\alpha)\\
:=&\bigcup_{\tilde{\alpha} \in \mathrm{passwords}(\beta_n,\alpha)}
\mathrm{passwords}(\beta_1\beta_2\ldots \beta_{n-1},\tilde{\alpha}).
\\
\vdots\\
 :=&\bigcup_{\tilde{\alpha} \in
\mathrm{passwords}(\beta_2\ldots \beta_n,\alpha)}
\mathrm{passwords}(\beta_1,\tilde{\alpha}).
\end{aligned}
$$
Moreover we let that $\mathrm{passwords}(\beta,\alpha)=\varnothing$
if $\beta\alpha$ is  not a reduced word and that
$\mathrm{passwords}(\beta,\alpha)=\{\alpha\}$ if $\beta$ is the
empty word. Finally, for any two sets $A$ and  $B$  of words, we
define
$$[ B,A]:=\bigcup_{\alpha\in A,\beta\in B} \mathrm{passwords}
(\beta,\alpha).$$
\end{defn}

\begin{ex} Consider $\alpha = 3456~ 78$, and $\beta= 96$.
Then
$$\begin{aligned}\mathrm{passwords}(6,3456~ 78)=&\{6~ 3456~
78,~ 3456~ 5~ 78 ,~3456~ 78~ 5~ \}\\
\mathrm{passwords}(9~6,3456~ 78)=& \{9~6~ 3456~ 78,~
9~3456~ 5~ 78 ,~9~3456~ 78~ 5~ \\
& ~~6~9~ 3456~ 78,~6~ 3456~9~ 78,\\
&~~9~3456~ 5~ 78,~3456~9~ 5~ 78,~3456~ 5~9~ 78,\\
& ~~9~3456~ 78~ 5,~3456~9~ 78~ 5 \}
\end{aligned}
$$
\end{ex}

 We are now ready to define basic words of a permutation.

\begin{defn} Let $\eta_w=\eta_1\ldots\eta_k$ be the tower
decomposition of the natural word of $\omega$ and let
$N_i=\{\eta_i\}$ for each $1\leq i\leq k$. Then the
 \textit{set of basic words} for $\omega$ is the set given by
 $$
\mathrm{Basic}(\omega) :=[[\ldots [[N_1,N_{2}], N_{3}],\ldots,
N_{k-1} ],N_k ].
$$
\end{defn}

The following result follows directly from the
above definition. We leave the justification to the reader.
\begin{lem}\label{lem:basicProp} Let $\omega $ be a permutation. Then any basic word in $\mathrm{Basic}(\omega) $ is reduced and is braid related to
the natural word of $\omega$.
\end{lem}
\subsection{Restricted shuffle}\label{Subsection:restrictedshuflle}

The second step of the generation algorithm is the restricted shuffle. This operation is a restriction  of the well-
known shuffle operation on words. Recall that, given two words $\alpha$ and $\beta$, a shuffle of $\beta$ over
$\alpha$ is obtained  by placing the letters of $\beta$ arbitrarily between the letters of $\alpha$ without changing
the order of letters of $\beta$. The set of all shuffles of $\beta$ over $\alpha$, denoted by $\mathrm{Sh}(\alpha,
\beta)$ can be obtained by  first concatenating  $\beta$ to the right of $\alpha$  to obtain a new word $\alpha\,
\beta$ and then moving the letters of $\beta$ to the left, without changing their orders, until  the word $\beta\,
\alpha$ is obtained.

On the other hand, the restricted shuffle employs the same idea by
adding a restriction: A letter $\beta_i$ of $\beta$ can pass a
letter $\alpha_j$ of $\alpha$ if and only if $\alpha_j\notin \{
\beta_i-1, \beta_i,\beta_i+1 \}$. With this definition, it is clear
that we are referring to the short braid relation. More precise
definition is as follows.

\begin{defn}
Suppose that the letters in  $\alpha=\alpha_1\ldots\alpha_n$ and
$\beta=\beta_1\beta_2\cdots \beta_m$ are colored by red and blue,
respectively. A \textit{restricted shuffle} of $\alpha$ with $\beta$
is a word $w=w_1\ldots w_{n+m}$ of $n$ red and $m$ blue letters
satisfying the following conditions.
\begin{enumerate}
\item The restrictions of $w$ on the red and blue letters give the words $\alpha$ and $\beta$ respectively.
\item If, for some $1\leq k\leq m$ and
 $1\leq i\leq n$, the letter $\beta_k$ lies to the left
$\alpha_i$ in  $w$   then none of $\{\beta_k-1,\beta_k,\beta_k+1\}$
lies in $\alpha_i\ldots \alpha_n $.
\end{enumerate}
We denote by $\mathrm{ResSh}(\alpha,\beta)$ the set of all restricted shuffles of $\alpha$ with $\beta$.
\end{defn}

\begin{ex} Let $\alpha = 13425$ and $\beta = 37$ and  color the word $\beta$ by
boldface. Then
$$\mathrm{ResSh}(\alpha,\beta)=\{
13425 \textbf{37}, 1342\textbf{3}5\textbf{7}, 1342\textbf{37} 5\}.
$$
For  $\alpha = 13465$ and $\beta = 37$ we have
$$\mathrm{ResSh}(\alpha,\beta)=\{
13465 \textbf{37}, 1346\textbf{3}5\textbf{7}, 1346\textbf{37} 5,
134\textbf{3}65\textbf{7},  134\textbf{3}6\textbf{7}5\}.
$$
\end{ex}

The following result follows easily from the definition of the restriction shuffle and the definitions of the braid
relations. We leave the straightforward proof to the reader.

\begin{lem}\label{lem:basicResSh}
Let $\alpha$,  $\beta$ and the concatenation  $\alpha\beta$ be
reduced words of  some permutations. Then
\begin{enumerate}
\item[i.] any restricted shuffle of $\alpha$ with $\beta$ is also reduced and
\item[ii.] the restricted shuffles of $\alpha$ with $\beta$ are pairwise distinct.
\end{enumerate}
\end{lem}

As in the case of passage words, we can generalize the above
operation to a restricted shuffle of several words by induction. Let
$u_1,u_2,\ldots,u_n$ be some words. Then we define
\[
\mathrm{ResSh}(u_1,u_2,\ldots,u_n) :=
\bigcup_{\alpha\in\mathrm{ResSh}(u_1,u_2,\ldots u_{n-1})}
\mathrm{ResSh}(\alpha,u_n).
\]
\subsection{Generation Theorem}

We are now ready to state the generation theorem

\begin{thm}[Generation Theorem]\label{thm:generation}
Let  $\omega$ be a permutation. Then
\[
\mathrm{Red}(\omega)= \bigcup_{\alpha\in \mathrm{Basic}(\omega)}
\mathrm{ResSh}(\mathfrak{a}_1,\mathfrak{a}_2,\ldots,\mathfrak{a}_r)
\]
where $\alpha\in\mathrm{Basic}(\omega)$  is given by  its tower
decomposition $\alpha =
\mathfrak{a}_1\mathfrak{a}_2\ldots\mathfrak{a}_r$.
\end{thm}

\begin{proof}
It is clear from Lemma \ref{lem:basicResSh} and Lemma
\ref{lem:basicProp} that the right hand side is contained in the
left hand side. To prove the converse inclusion, it is sufficient to
show that the sorting algorithm is inverse to the generation
algorithm. But clearly, any step in the selection sort algorithm  is
a restricted shuffle. Moreover, these steps are consistent with the
order of the parenthesis in the restricted shuffle of the generation
algorithm. Similarly, any natural basic word is a basic word. Indeed
the reverse of each step in the insertion sort algorithm is a
passage word construction.
\end{proof}
\section{Example: Longest words}\label{sec:ex}
As our first example, we produce all reduced expressions for the longest word in $S_4$.
 The general case, for an arbitrary $S_n$, can be treated in the same way.
Recall that the longest word $\omega_0$ is the reverse of the
identity permutation and  its natural word is given by  $\eta =
3\,\, 23\,\, 123$ in its tower decomposition. Then
$$\mathrm{Basic}(\omega_0)=[[\{3\},\{23\}],\{123\} ].
$$
Here $[ \{3\},\{23\}]= \mathrm{passwords}(\mb{3}, 23) = \{ \mb{3}\,
23,23\, \mb{2} \}$ and hence
$$\mathrm{Basic}(\omega_0)=
\mathrm{passwords}(\mb{3\, 23}, 123)\cup \mathrm{passwords}(\mb{23\,
2}, 123) $$ where
$$\begin{aligned}
\mathrm{passwords}(\mb{3\, 23}, 123) =& \{ \mb{3\, 23}\, 123, \mb{3\, 2}\, 123\, \mb{2}, \mb{3}\, 123\, \mb{12} ,123\, \mb{2\, 12}, 123\, \mb{12 \, 1} \}\\
\mathrm{passwords}(\mb{23\, 2}, 123) =& \{ \mb{23\, 2}\, 123, \mb{23}\, 123\, \mb{1}, \mb{2}\, 123\, \mb{2\,1} ,123\, \mb{12\, 1} \}.
\end{aligned}$$
 Note that the last
elements of the above sets coincide and we omit one of them. To
obtain the set of all reduced words, it remains to apply restricted
shuffle to each of the basic words which are listed below.
\begin{eqnarray*}
\mathrm{ResSh}(3,23,123) &=& \{ 3\, 23\, \mb{123}, 3\, 2\, \mb{1}\, 3\, \mb{23} \}\\
\mathrm{ResSh}(3,2,123, 2) &=& \{ 3\, 2\, 123 \, 2\}\\
\mathrm{ResSh}(3,123, 12) &=& \{ 3\, \mb{123}\, 12, \mb{1} \, 3\, \mb{23}\, 12,  3\, \mb{12} \, 1 \, \mb{3}\, 2, \mb{1} \, 3\,
\mb{2}\, 1\, \mb{3}\, 2\}
\end{eqnarray*}

\begin{eqnarray*}
\mathrm{ResSh}(123,2,12) &=& \{ 123\, 2\, 12\}\\
\mathrm{ResSh}(23,2,123) &=& \{ 23\, 2\, 123\}\\
\mathrm{ResSh}(23,123, 1) &=& \{ 23\, \mb{123}\, 1, 2\, \mb{1} \, 3\, \mb{23}\, 1, 23 \, \mb{12}\, 1\, \mb{3}, 2\, \mb{1}\, 3\,
\mb{2}\, 1\, \mb{3} \}\\
\mathrm{ResSh}(2,123,2,1) &=& \{ 2\, 123\, 2\, 1\}\\
\mathrm{ResSh}(123,12,1) &=& \{ 123\, \mb{12}\, 1, 12\,\mb{1} \,3 \,\mb{2}\, 1    \}
\end{eqnarray*}
Finally, the union of all these restricted shuffles gives us the
full set of reduced words for the longest permutation of $S_4$. Note
that there are $16$ reduced words in the union and the number
coincides with the one that Stanley's formula \cite{S} gives.

\end{document}